\documentclass[12pt]{amsart}

\usepackage{amsthm,amsfonts}
\usepackage{amssymb,graphicx,color}

\newtheorem{theorem}{Theorem}[section]

\newtheorem{remark}[theorem]{Remark}

\newcommand{\R}{\mathbb R}

\begin{document}

\title[Lipschitz knot theory]{Surface singularities in $\R^4$: first steps towards Lipschitz knot theory}

\author{Lev Birbrair*}\thanks{*Research supported under CNPq 302655/2014-0 grant and by Capes-Cofecub}
\address{Dept Matem\'atica, Universidade Federal do Cear\'a
(UFC), Campus do Picici, Bloco 914, Cep. 60455-760. Fortaleza-Ce,
Brasil} \email{birb@ufc.br}

\author{Andrei Gabrielov**}\thanks{**Research supported by the NSF grant DMS-1665115}
\address{Dept Mathematics, Purdue University, 150 N. University Street, West Lafayette, IN 47907-2067, USA } \email{gabriea@purdue.edu}

\begin{abstract}
A link of an isolated singularity of a two-dimensional semialgebraic surface in $\R^4$ is a knot (or a link) in $S^3$. Thus the ambient Lipschitz classification of surface singularities in $\R^4$ can be interpreted as
a metric refinement of the topological classification of knots (or links) in $S^3$.
We show that, given a knot $K$ in $S^3$, there are infinitely many distinct ambient Lipschitz equivalence classes of outer Lipschitz equivalent singularities in $\R^4$ with the links topologically ambient equivalent to $K$.
\end{abstract}

\maketitle

\section{Introduction}

here are three kinds of equivalence relations in Lipschitz Geometry of Singularities. One equivalence relation, inner Lipschitz equivalence, is bi-Lipschitz homeomorphism (of the germs at the origin) of singular sets with respect to the inner metric, where the distance between two points of a set $X$ is defined as infimum of the lengths of paths inside $X$ connecting the two points. The second equivalence relation, outer Lipschitz equivalence, is bi-Lipschitz homeomorphism with respect to the outer metric, where the distance is defined as the distance between the points in the ambient space.
A set $X$ is called normally embedded if its inner and outer metrics are equivalent.

 In \cite{BirbrairGabrielov2019}, we considered the third equivalence relation, ambient Lipschitz equivalence. Two germs $X$ and $Y$ of semialgebraic sets at the origin of $\R^n$ are called ambient Lipschitz equivalent if there exists a germ of a bi-Lipschitz homeomorphism $h$ of $(\R^n,0)$ such that $Y=h(X)$. In particular, such sets $X$ and $Y$ are outer Lipschitz equivalent. Two outer Lipschitz equivalent sets are always inner Lipschitz equivalent, but can be ambient topologically non-equivalent (see Neumann-Pichon \cite{NeumannPichon2017}).

Let $X$ and $Y$ be two semialgebraic surface singularities (two-dimensional germs at the origin) in $\R^n$ which are outer Lipschitz equivalent. Suppose also that $X$ and $Y$ are topologically ambient equivalent. Does it imply that the sets $X$ and $Y$ are ambient Lipschitz equivalent? It seems plausible that the answer is ``yes'' when $n\ge 5$, or when $X$ and $Y$ are normally embedded. However, examples in \cite{BirbrairGabrielov2019} show that the answer may be ``no'' when $n=3$ or $n=4$.

One class of examples in $\R^3$ and $\R^4$ is based on the theorem of Sampaio \cite{Sampaio2016}: ambient Lipschitz equivalence of two sets implies ambient Lipschitz equivalence of their tangent cones. Thus any two sets with topologically ambient non-equivalent tangent cones cannot be ambient Lipschitz equivalent.

The case $n=4$ is especially interesting, as in that case the link of a two-dimensional germ
 $X$ in $\R^4$ is a knot (or a link) in $S^3$, and the arguments are based on the knot theory.
 For a given surface $X\subset \R^3$ there are finitely many distinct ambient Lipschitz equivalence classes of the
 surfaces which are topologically ambient equivalent and outer Lipschitz equivalent to $X$. However, there may be infinitely many such ambient Lipschitz equivalence classes for a surface in $\R^4$. Moreover, a more delicate argument, based on the ``bridge construction'' below, provides infinitely many distinct Lipschitz ambient equivalence classes of surfaces which are topologically ambient equivalent to a given surface $X\subset\R^4$ and all belong to the same outer Lipschitz equivalence class, even when each of these surfaces has a tangent cone consisting of a single ray.

 In this paper we use the ``topological ambient equivalence'' relation for the germs at the origin of singular sets, and for their links.
 It means that there exists a homeomorphism of a small ball (or a small sphere) centered at the origin, mapping one singular set (or its link) to another.
 This definition corresponds to the classical topological equivalence in Singularity Theory.
 For the links of surfaces in $\R^4$ this equivalence is closely related to isotopy of knots. There is, however, a minor difference between the topological ambient equivalence and isotopy: in Knot Theory the homeomorphism is required to be isotopic to identity.
 Two knots may be topologically ambient equivalent but not isotopic.
 This difference between two equivalence relations is non-essential, as the number of ambient Lipschitz equivalence classes is infinite in both cases.

The authors thank the anonymous referee whose thoughtful suggestions helped us to improve the exposition.
\begin{figure}
 \centering
 \includegraphics[width=4in]{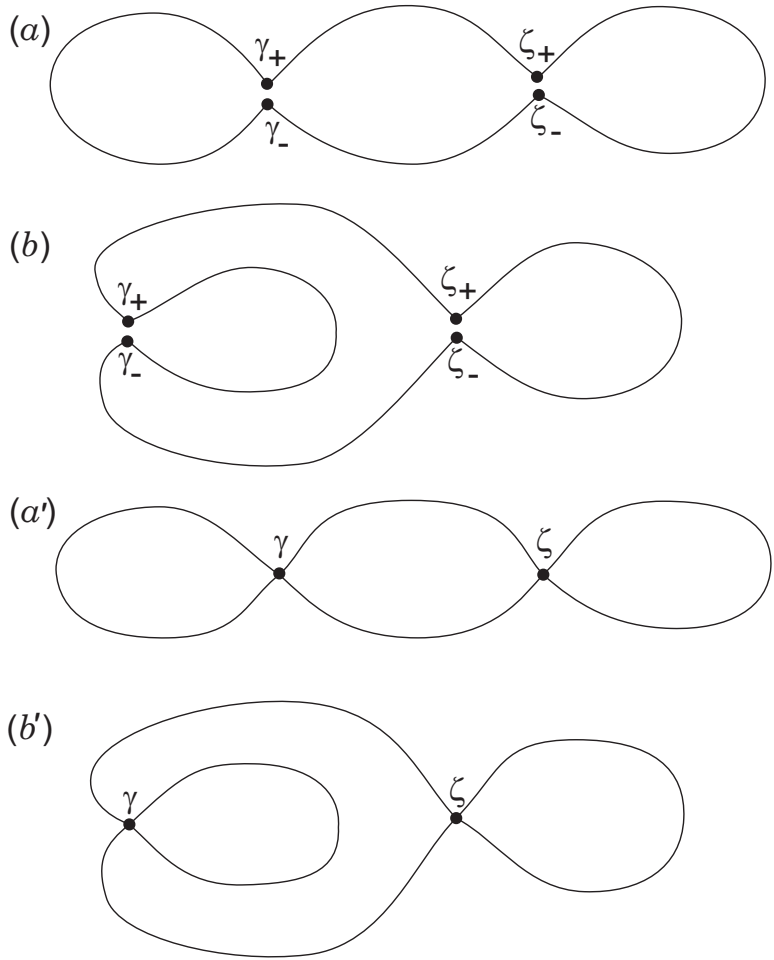}
 \caption{The links of the surfaces $X_1$ and $X_2$ in Example 1,
 and of their tangent cones.}\label{fig:example2b}
 \bigskip
 \end{figure}
 \begin{figure}
 \centering
 \includegraphics[width=4in]{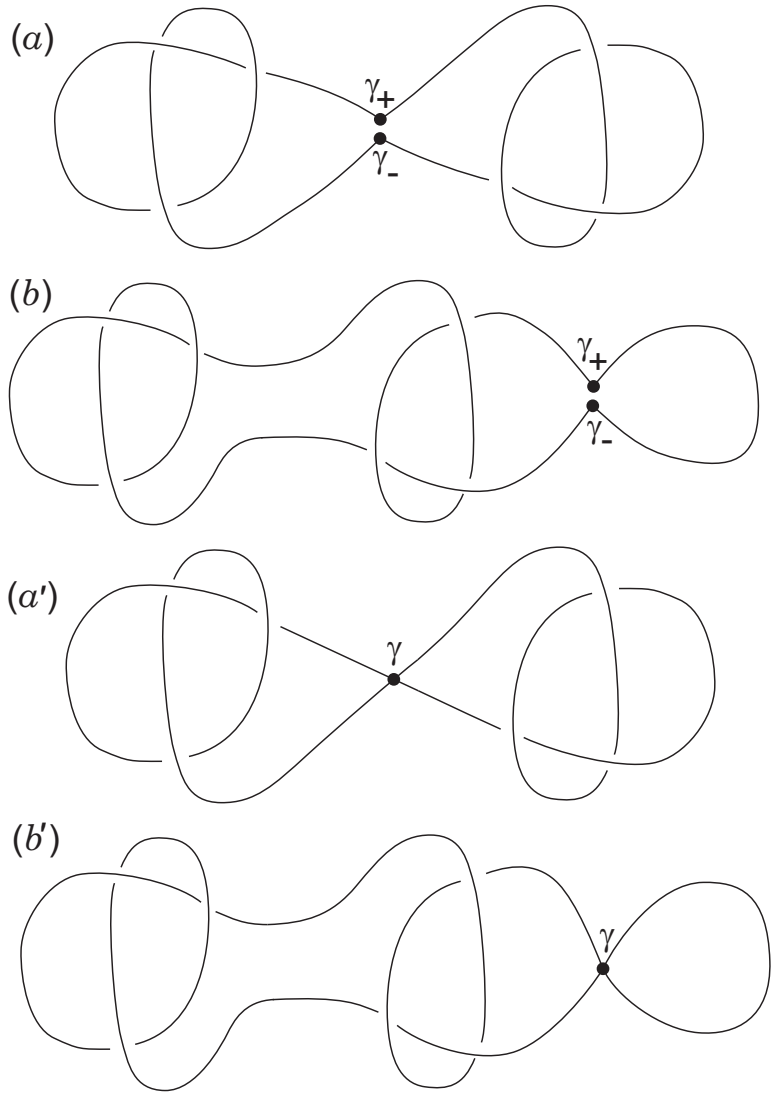}
 \caption{The links of the surfaces $X_1$ and $X_2$ in Example 2,
 and of their tangent cones.}\label{fig:knot1bt}
 \bigskip
 \end{figure}

\section{Examples in $\R^3$ and $\R^4$ based on Sampaio's theorem}

In this paper, an {\em arc} in $\R^n$ is a germ at the origin of a semialgebraic mapping $\gamma:[0,\epsilon)\to\R_n$ such that $\|\gamma(t)\|=t$.
The {\em tangency order} $\text{tord}(\gamma_1,\gamma_2)$ of two arcs is the smallest Puiseux exponent at zero of the function $\|\gamma_2(t)-\gamma_1(t)\|$.
Sometimes in the literature the tangency order is called the tangency exponent or the order of contact (see \cite{BFGG}, \cite{BM}, \cite{Parusinski1988}).

\noindent{\bf Example 1} (see \cite{BirbrairGabrielov2019}, Example 2). Let $X_1$ and $X_2$ be two surfaces in $\R^3$ with the links at the origin shown in Fig.~\ref{fig:example2b}$a$ and Fig.~\ref{fig:example2b}$b$. There are two pairs of ``pinched'' arcs, $\gamma_\pm$ and $\zeta_\pm$,
with $\text{tord}(\gamma_+,\gamma_-)>1$ and $\text{tord}(\zeta_+,\zeta_-)>1$, while both surfaces are straight-line cones over their links outside small conical neighborhoods of $\gamma_\pm$ and $\zeta_\pm$. The arcs $\gamma_+$ and $\gamma_-$ correspond to a single ray $\gamma$ of the tangent cone, and the arcs $\zeta_+$ and $\zeta_-$ correspond to a single ray $\zeta$. One can define $X_1$ and $X_2$ by explicit semialgebraic formulas. For small conical neighborhoods of the pinched arcs $\gamma_\pm$ and $\zeta_\pm$ it can be done as in \cite{BirbrairGabrielov2019}, Example 3, and outside those neighborhoods the links of $X_1$ and $X_2$ can be approximated by piecewise linear curves.
Both surfaces $X_1$ and $X_2$ are topologically ambient equivalent to a cone over a circle in $S^2$ and outer Lipschitz equivalent, but not ambient Lipschitz equivalent by Sampaio's theorem, since their tangent cones are not topologically ambient equivalent: there is a connected component of the complement in $S^2$ of the link of $X_1$ (see Fig.~\ref{fig:example2b}$a'$) with the whole link as its boundary, while there is no such component of the complement of the link of $X_2$ (see Fig.~\ref{fig:example2b}$b'$).
\bigskip

\noindent{\bf Example 2} (see \cite{BirbrairGabrielov2019}). Let $X_1$ and $X_2$ be two surfaces in $\R^4$ with the links at the origin shown in Fig.~\ref{fig:knot1bt}$a$ and Fig.~\ref{fig:knot1bt}$b$, and the links of their tangent cones at the origin
shown in  and Fig.~\ref{fig:knot1bt}$b'$. The tangency exponent of the arcs
$\gamma_+$ and $\gamma_-$ is $\alpha>1$, thus the arcs $\gamma_+$ and $\gamma_-$ correspond to a single ray $\gamma$ of the tangent cone. One can define $X_1$ and $X_2$ by explicit semialgebraic formulas.
Both surfaces $X_1$ and $X_2$ are topologically ambient equivalent to a cone over a circle in $S^3$ embedded as the connected sum of two trefoil knots, and outer Lipschitz equivalent but not ambient Lipschitz equivalent by Sampaio's theorem, since their tangent cones at the origin are not topologically ambient equivalent: the link of the tangent cone of $X_1$ (see Fig.~\ref{fig:knot1bt}$a'$) is a bouquet of two knotted circles, while the link of $X_2$ (see Fig.~\ref{fig:knot1bt}$b'$) is a bouquet of a knotted circle and an unknotted one.

\section{Bridge construction}

A $(q,\beta)$-\emph{bridge} is the set $A_{q,\beta}=T_+\cup T_-\subset {\mathbb R}^4$ where $1<\beta<q$ and
$$T_\pm=\left\{ 0\le t\le 1,\; -t^\beta\le x\le t^\beta,\; y=\pm t^q,\; z=0\right\}.$$
Its link is shown in Fig.~\ref{fig:bridge} (left).
A \emph{broken $(q,\beta)$-bridge} $B_{q,\beta}$ is obtained from $A_{q,\beta}$ by the \emph{saddle operation}, removing from $T_\pm$ two $p$-H\"older triangles
$$\left\{t\ge 0,\; |x|\le t^p,\; y=\pm t^q,\; z=0\right\}$$
where $p>q$, and replacing them by two $q$-H\"older triangles
$$\left\{ 0\le t\le 1,\; x=\pm t^p,\; |y|\le t^q,\; z=0\right\}.$$
Its link is shown in Fig.~\ref{fig:bridge} (right).
We call $(q,\beta)$-bridge any surface outer Lipschitz equivalent to $A_{q,\beta}$.
It was shown in \cite{BirbrairGabrielov2019} that ambient Lipschitz equivalence $h:X\to Y$ of two surfaces in $\R^4$ maps a $(q,\beta)$-bridge in $X$ to a $(q,beta)$-bridge in $Y$, and that the two surfaces remain ambient Lipschitz equivalent when their $(q,\beta)$-bridges are replaced by the broken $(q,\beta)$-bridges.

\begin{remark}\label{broken} Our definition of a broken bridge is slightly different from the definition in Example 4 of \cite{BirbrairGabrielov2019}, where it was defined with $p<q$.
Condition $p>q$ makes the ``broken bridge'' operation invertible:
two surface germs with the same $(q,\beta)$-bridge are ambient Lipschitz equivalent if and only if they are ambient Lipschitz equivalent after the bridge is broken (with the same $p>q$). Note that this invertibility is never used here or in \cite{BirbrairGabrielov2019}.
\end{remark}

\begin{figure}
 \centering
 \includegraphics[width=4.5in]{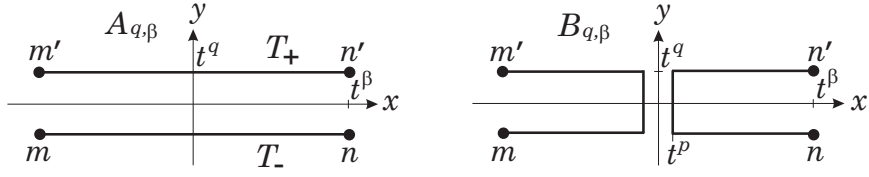}
 \caption{The links of a $(q,\beta)$-bridge $A_{q,\beta}$ and a broken $(q,\beta)$-bridge $B_{q,\beta}$.}\label{fig:bridge}
 \end{figure}

\begin{figure}
 \centering
 \includegraphics[width=4.5in]{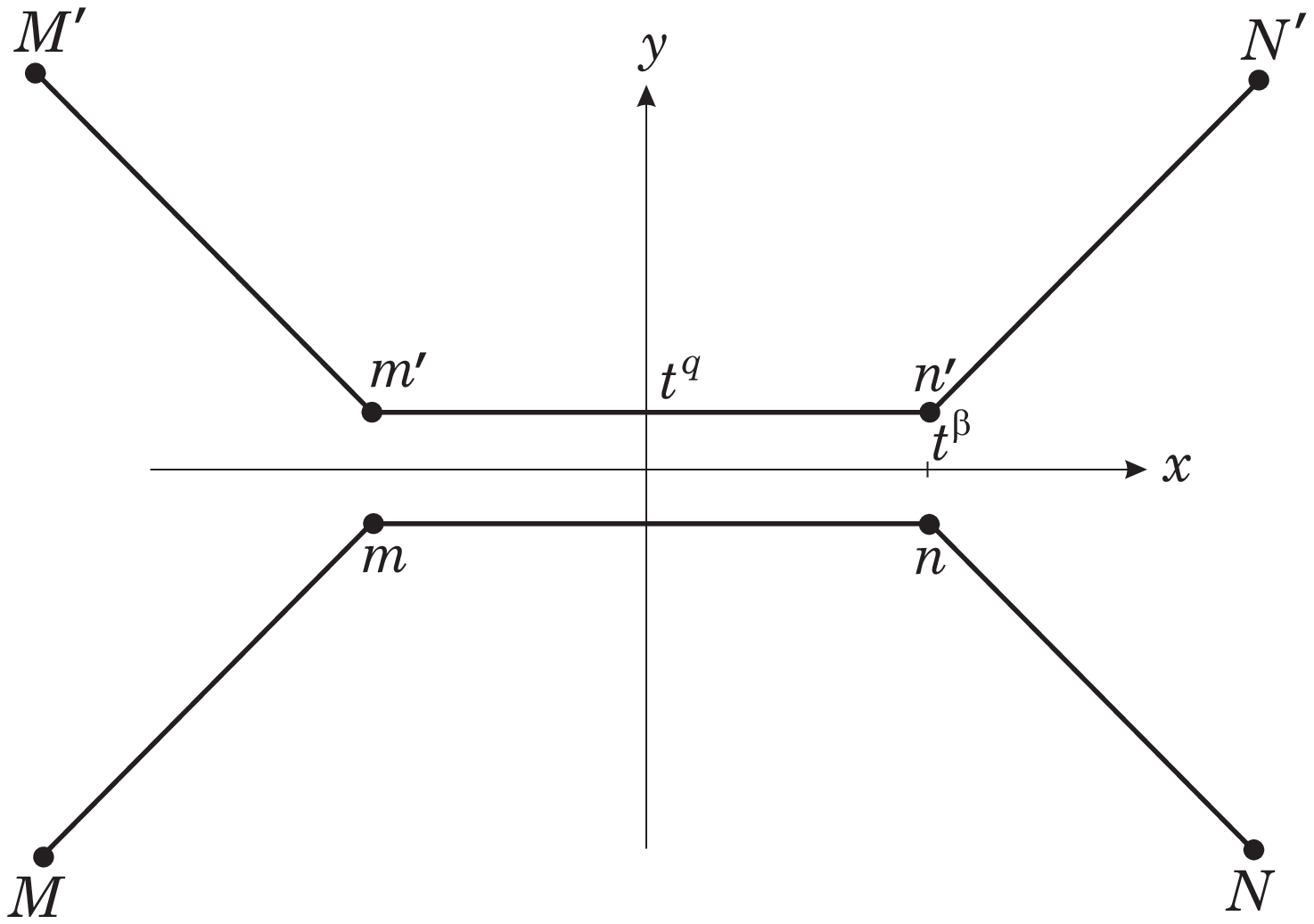}
 \caption{The link of the surface $G$ in Example 3.}\label{fig:G}
 \end{figure}

\begin{figure}
 \centering
 \includegraphics[width=4.5in]{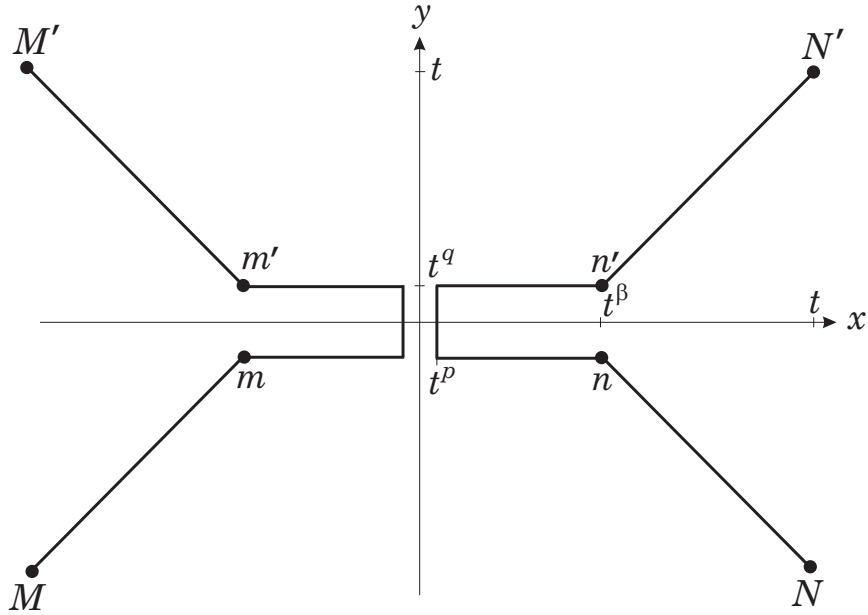}
 \caption{The link of the surface $H$ in Example 3.}\label{fig:H}
 \end{figure}

\noindent{\bf Example 3} (see \cite{BirbrairGabrielov2019}).
The common boundary of $A_{q,\beta}$ and $B_{q,\beta}$ consists of the four arcs $\left\{0\le t\le 1,\; x=\pm t^\beta,\; y=\pm t^q,\; z=0\right\}$ shown as $m,\,n,\,m',\,n'$ in Fig.~\ref{fig:bridge}.
Let $G \subset \R^4$ be a semialgebraic surface containing $A_{q,\beta}$ and bounded by the four straight line segments $\{0\le t\le 1,\; \pm x =\pm y = t,\;z=0\}$ (see Fig.~\ref{fig:G} where the boundary arcs of $G$ are shown as $M,\,N,\,M',\,N'$). Let $H$ be the surface obtained from $G$ by replacing the bridge $A_{\beta,q}$ by the broken bridge $B_{\beta,q}$ (see Fig.~\ref{fig:H}).

Consider two topologically trivial knots $K$ and $L$ in the hyperplane $\{t=1\}\subset \R^4_{x,y,z,t}$ as shown in Fig.~\ref{fig:knot2}a and Fig.~\ref{fig:knot2}b.
Each of these two knots contains the curve $g=G\cap\{t=1\}$,
We define the surface $X$ as the union of $G$ and a straight cone over $K\setminus g$, and the surface $Y$ as the union of $G$ and a straight cone over $L\setminus g$.

\begin{figure}
 \centering
 \includegraphics[width=4.5in]{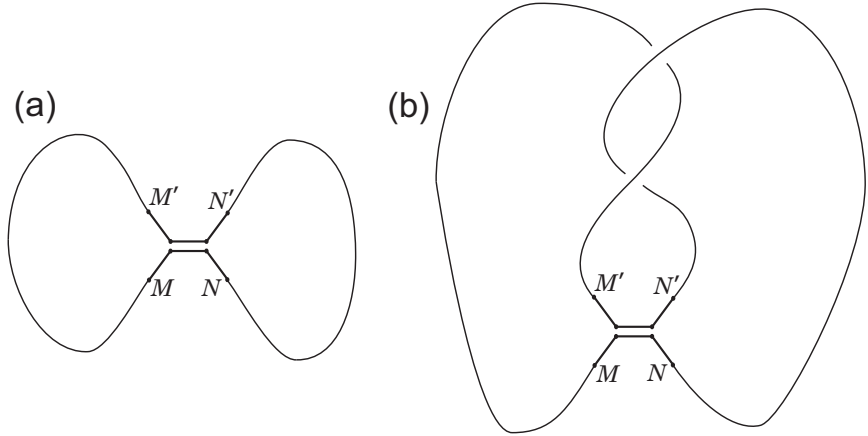}
 \caption{The links of the surfaces $X$ and $Y$ in Example 3.}\label{fig:knot2}
 \end{figure}

\begin{figure}
 \centering
 \includegraphics[width=4.5in]{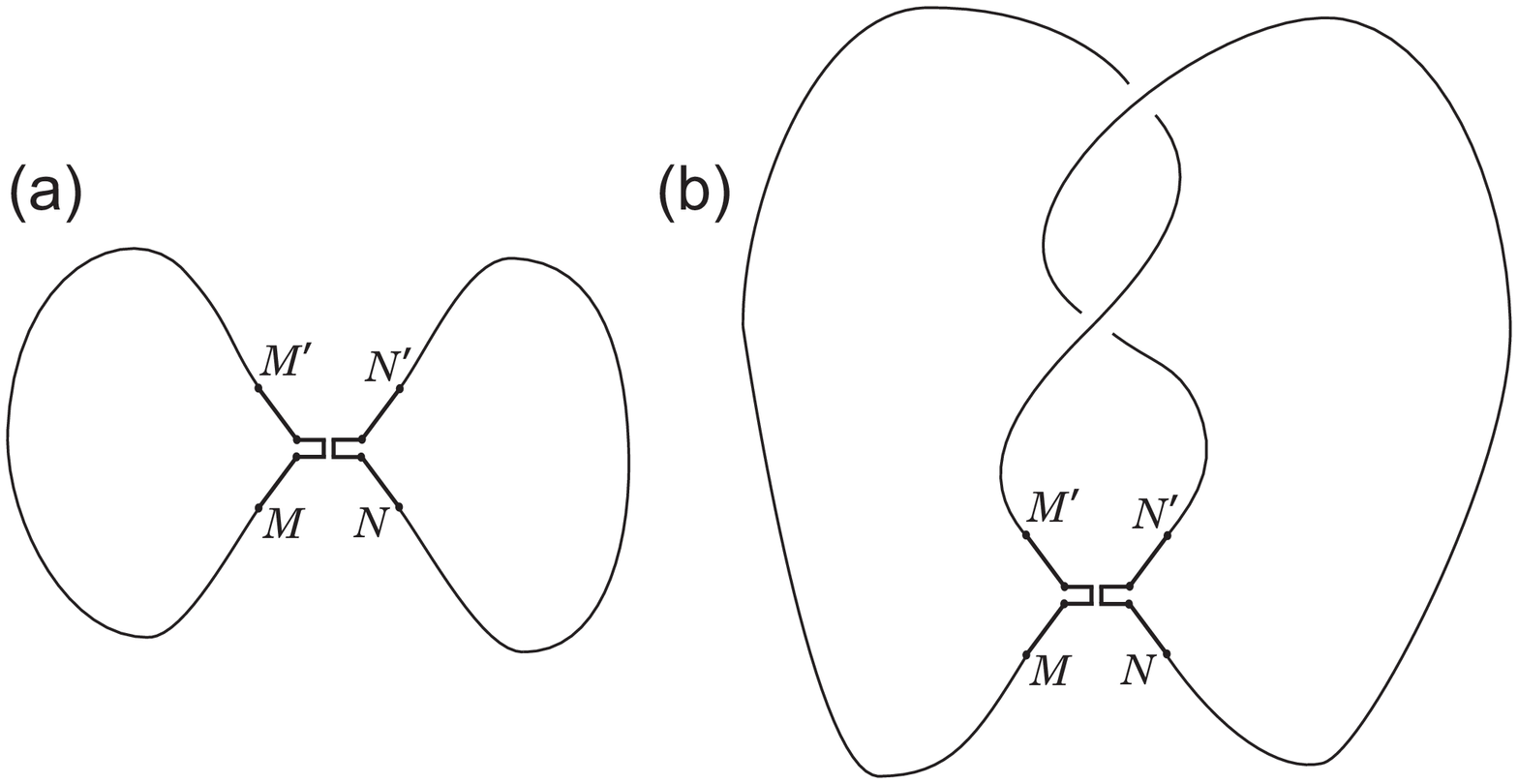}
 \caption{The links of the surfaces $X'$ and $Y'$ in Example 3.}\label{fig:knot2b}
 \end{figure}

\begin{figure}
 \centering
 \includegraphics[width=4.5in]{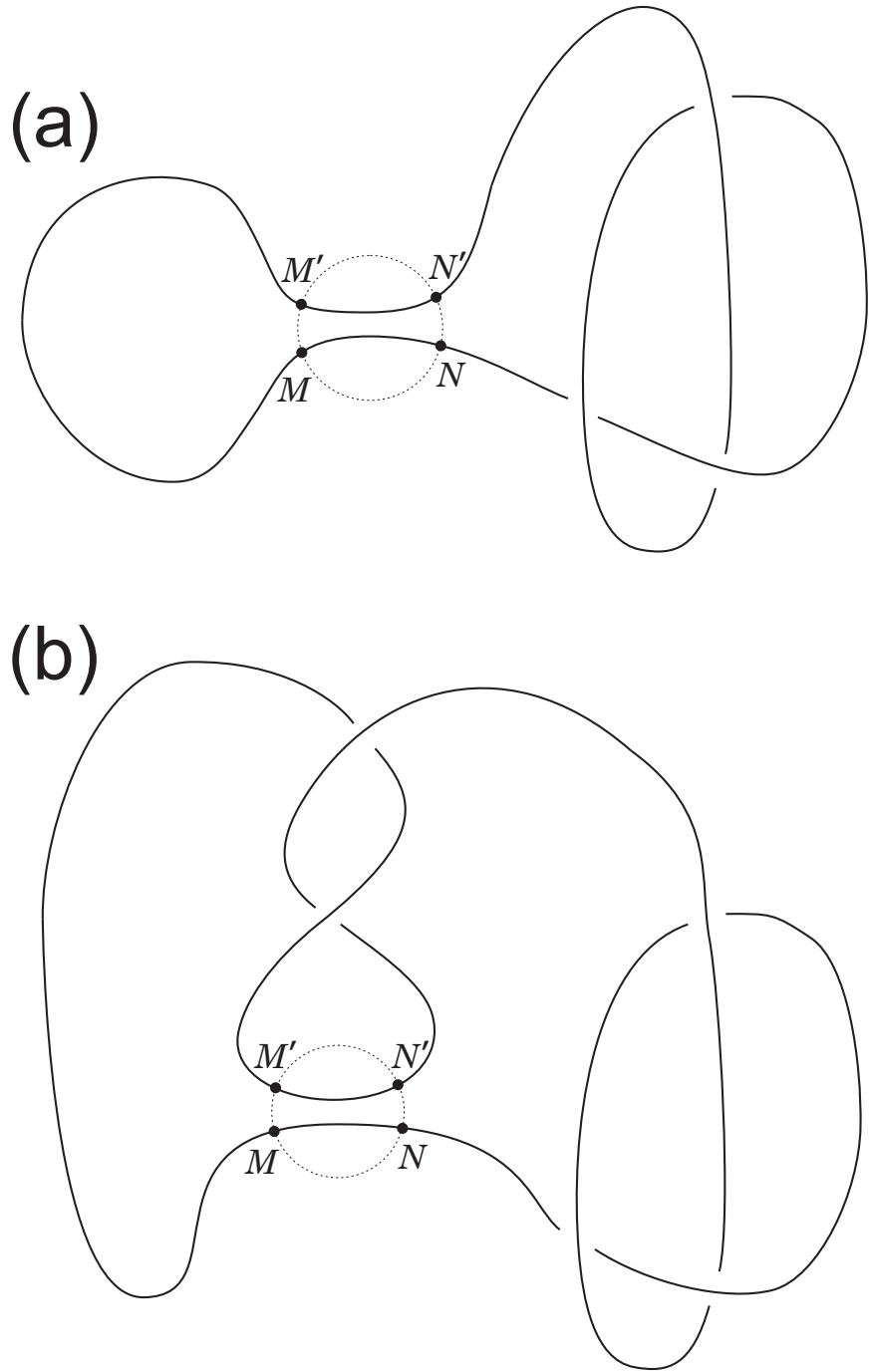}
 \caption{The links of the surfaces $X$ and $Y$ with an extra knot attached.}\label{fig:knot3}
 \end{figure}

\begin{figure}
 \centering
 \includegraphics[width=4.5in]{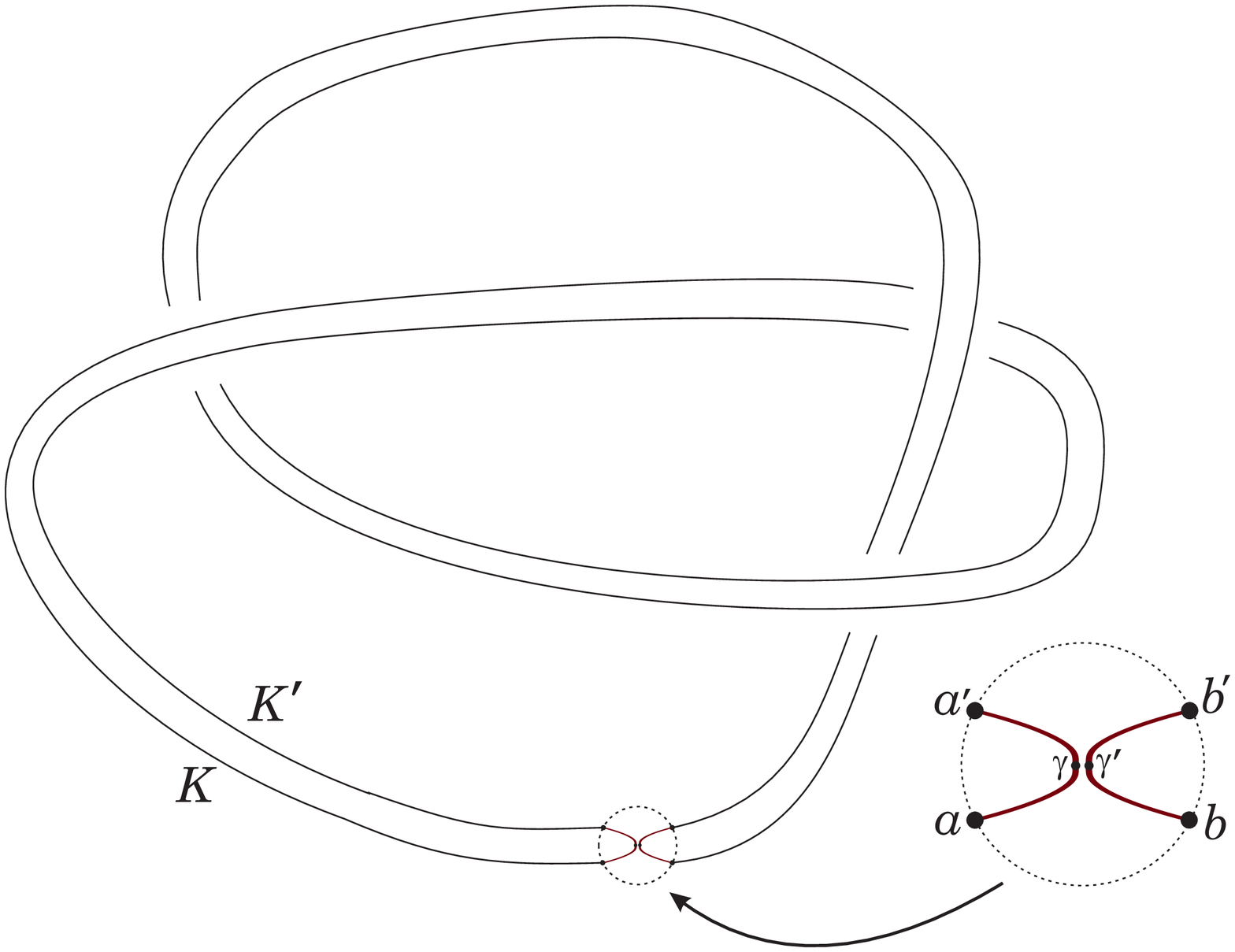}
 \caption{The link of the surface $X$ in Theorem \ref{universal}.}\label{doubleknot}
 \end{figure}

\begin{theorem}\label{thm:ex3} (see \cite{BirbrairGabrielov2019} Theorem 3.2.)
The germs of the surfaces $X$ and $Y$ at the origin are outer Lipschitz equivalent, topologically ambient equivalent, but not ambient Lipschitz equivalent.
\end{theorem}

This is proved by replacing the $(q,\beta)$-bridges in $X$ and $Y$ by the broken $(q,\beta)$-bridges, resulting in
the new surfaces $X'$ and $Y'$, shown in Fig.~\ref{fig:knot2b}. The link of $X'$ consists of two unlinked circles while the two circles in the link of $Y'$ are linked. Thus $X'$ and $Y'$ are not topologically ambient equivalent, which
implies that $X$ and $Y$ are not ambient Lipschitz equivalent.

\begin{remark}
Notice that the tangent cones of both surfaces $X$ and $Y$ in Example 3 are topologically ambient equivalent to a cone over two unknotted circles, pinched at one point.
Thus Sampaio's theorem does not apply, and we need the bridge construction in this example.
Notice also that the bridge construction employed in this example allows one to construct examples
 of outer Lipschitz equivalent, topologically ambient equivalent
but ambient Lipschitz non-equivalent surface singularities in $\mathbb{R}^4$ with the tangent cones as small as a single ray.
\end{remark}

The surfaces $X$ and $Y$ in Example 3 differ by a ``twist'' of the $(q,\beta)$-bridge, which can be extended to a homeomorphism of the ambient space, but not to a bi-Lipschitz homeomorphism.
One can iterate such a twist to obtain infinitely many ambient Lipschitz non-equivalent surfaces.
 On can also attach an additional knot to the links of both surfaces $X$ and $Y$ (see Fig.~\ref{fig:knot3}).
 This yields the following ``universality'' result (see \cite{BirbrairGabrielov2019} Theorem 4.1).

\begin{theorem}\label{thm:main} For any semialgebraic surface germ $S \subset \R^4$ there exist infinitely many semialgebraic surface germs $X_i \subset \R^4$ such that  \vspace{0.05in}

$1) \ $ For all  $i$, the germs $X_i$ are topologically ambient equivalent to $S$;

$2) \ $ All germs $X_i$ are outer Lipschitz equivalent;

$3) \ $ The tangent cones of all germs $X_i$ at the origin are topologically ambient equivalent;

$4) \ $ For $ i \neq j$ the germs $X_i$ and $X_j$ are not ambient Lipschitz equivalent.
\end{theorem}

Other versions of universality can be formulated. In particular, the following theorem is proved in \cite{BirbrairBrandenbursky2020}.

\begin{theorem}\label{universal} \emph{(Universality Theorem)}
For each knot $K\subset S^3$, there exists a germ at the origin of a semialgebraic surface $X_K\subset\R^4$
such that \vspace{0.05in}

$1) \ $ The germs $X_K$ are outer Lipschitz equivalent for all knots $K$.

$2) \ $ The links of all germs $X_K$ are topologically trivial knots in $S^3\subset\R^4$.

$3) \ $ The germs $X_{K_1}$ and $X_{K_2}$ are ambient Lipschitz equivalent only if the knots $K_1$ and $K_2$ are topologically ambient equivalent.
\end{theorem}

\begin{proof}
Idea of the proof is illustrated in Fig.~\ref{doubleknot}. The knot $K$ is realized as a smooth semialgebraic circle embedded in $S^3$.
Let $L$ be a semialgebraic strip homeomorphic to $K\times[0,\epsilon]$ embedded in $S^3$ so that $K$ is one of the two boundary curves of $L$.
Let $K'$ be the other boundary curve of $L$, so that the knots $K$ and $K'$ are isotopic.
Let $C\subset\R^4$ be a cone over $K\cup K'$ with the vertex at the origin.
Let $y$ be a point such that the intersection of a small ball $B\subset S^3$ centered at $y$ with $K\cup K'$
consists of two segments, $[a,b]\subset K$ and $[a',b']\subset K'$.
We can remove from $C$ the cone $C_B$ over $B\cap(K\cup K')$ and replace it by a semialgebraic surface $G$ inside the cone over $B$, with the link shown in the zoomed part of Fig.~\ref{doubleknot}, such that\newline
a) the link of $G$ consists of two segments $[a,a']$ and $[b,b']$,\newline
b) the boundary of $G$ is the same as the boundary of $C_B$,\newline
c) two arcs $\gamma$ and $\gamma'$ in $G$ have the tangency order $\beta>1$.\newline
This results in the surface $X_K$ with the link shown in Fig.~\ref{doubleknot}.
The link of $X_K$ is a trivial knot (it is contractible to a point inside $L$).
The link of the tangent cone of the surface $X_K$ consists of two knots, each of them isotopic to $K$, pinched at a point corresponding
to the common tangent line to $\gamma$ and $\gamma'$.
Thus the tangent cones of the surfaces $X_K$ are not topologically ambient equivalent for the knots which are not topologically ambient equivalent.
It follows from Sampaio's theorem that the same is true for the surfaces $X_K$ themselves.
\end{proof}

\begin{remark}\label{rmk:universal} \emph{Theorem \ref{universal} is called ``Universality Theorem'' because it implies that
the ambient Lipschitz classification problem for the surface germs in $\R^4$
in a single outer Lipschitz equivalence class,
with the topologically ambient trivial links, contains all of the Knot Theory.}
\end{remark}


\begin{thebibliography}{99}

\bibitem[BBG20]{BirbrairBrandenbursky2020}
Birbrair, L., Brandenbursky, M., Gabrielov, A.: Lipschitz geometry of surface germs in $\R^4$ and knot theory.
 In preparation (2020).

\bibitem[BFGG17]{BFGG} Birbrair, L., Fernandes, A., Gabrielov, A., Grandjean, V.: Lipschitz contact equivalence of function germs in $R^2$.
 Annali SNS Pisa, 17 (2017), 81--92.

\bibitem[BG19]{BirbrairGabrielov2019}
Birbrair, L., Gabrielov, A.: Ambient {L}ipschitz equivalence of real surface singularities.
 Int. Math. Res. Not. IMRN 20 (2019), 6347--6361.
doi{10.1093/imrn/rnx328}.

\bibitem[BM]{BM} Birbrair, L., Mendes, R.: Arc criterion of normal embedding.
 In: Singularities and Foliations. Geometry, Topology and Applications. NBMS 2015, BMMS 2015.
 Springer Proceedings in Mathematics \& Statistics, v. 222 (2018), p. 549--553.

\bibitem[Fer03]{F} Fernandes, A.: Topological equivalence of complex curves and bi-Lipschitz maps.
 The Michigan Mathematical Journal, Michigan, v. 51 (2003), p. 593--606.

\bibitem[NP17]{NeumannPichon2017}
Neumann, W.D., Pichon, A.: Lipschitz geometry does not determine embedded topological type.
 In: Singularities in geometry, topology, foliations and dynamics,
  Trends Math., pp. 183--195. Birkh\"{a}user/Springer, Cham (2017).

\bibitem[Pa]{Parusinski1988} Parusinski, A.: Lipschitz properties of semi-analytic sets. Annales de l'Institut Fourier, v. 38 (1988), p. 189--213.

\bibitem[PT69]{PT} P.~Pham and B.~Teissier, Saturation Lipschitzienne d'une alg\`ebre analytique complexe et saturation de Zariski. Pr\'epublication Ecole Polytechnique 1969.

\bibitem[Sam16]{Sampaio2016}
Sampaio, J.E.: Bi-{L}ipschitz homeomorphic subanalytic sets have bi-{L}ipschitz
  homeomorphic tangent cones.
 Selecta Math. (N.S.) \textbf{22}(2), 553--559 (2016).
doi{10.1007/s00029-015-0195-9}.

\end{thebibliography}
\end{document}